\documentclass[12pt,reqno,draft]{amsart}
\usepackage{enumerate, latexsym, amsmath, amsfonts, amssymb, amsthm, graphicx, color, }
\def\pmod #1{\ ({\rm{mod}}\ #1)}
\def\Z{\Bbb Z}

\def\Q{\Bbb Q}

\def\F{\Bbb F}

\def\bg{\bigg}
\def\({\bg(}
\def\){\bg)}

\def\sgn{{\rm sgn}}

\def\Arg{{\rm Arg}}

\def\Ack{\medskip\noindent {\bf Acknowledgments}}
\theoremstyle{plain}
\newtheorem{theorem}{Theorem}

\newtheorem{lemma}{Lemma}

\theoremstyle{definition}

\theoremstyle{remark}

\makeatletter
\@namedef{subjclassname@2010}{%
  \textup{2010} Mathematics Subject Classification}
\makeatother
 \vspace{4mm}

\begin{document}
 \baselineskip=17pt
\hbox{}
\medskip
\title[Squares in $\mathbb{F}_{p^2}$ and permutations involving primitive roots]
{Squares in $\mathbb{F}_{p^2}$ and permutations involving primitive roots}
\date{}
\author[Hai-Liang Wu] {Hai-Liang Wu}
\thanks{2010 {\it Mathematics Subject Classification}.
Primary 11A15; Secondary 05A05, 11R18.
\newline\indent {\it Keywords}. quadratic residues, permutations, primitive roots, local fields.
\newline \indent Supported by the National Natural Science
Foundation of China (Grant No. 11971222).}

\address {(Hai-Liang Wu) Department of Mathematics, Nanjing
University, Nanjing 210093, People's Republic of China}
\email{{\tt whl.math@smail.nju.edu.cn}}

\begin{abstract}
Let $p=2n+1$ be an odd prime, and let
$\zeta_{p^2-1}$ be a primitive $(p^2-1)$-th root of unity in the algebraic closure $\overline{\mathbb{Q}_p}$ of $\mathbb{Q}_p$.
We let
$g\in\mathbb{Z}_p[\zeta_{p^2-1}]$ be a primitive root modulo $p\mathbb{Z}_p[\zeta_{p^2-1}]$
Let $\Delta\equiv3\pmod4$ be an arbitrary quadratic non-residue modulo $p$ in $\mathbb{\Z}$. By the Local
Existence Theorem of class field theory we have $\mathbb{Q}_p(\sqrt{\Delta})=\mathbb{Q}_p(\zeta_{p^2-1})$.
For all $x\in\mathbb{Z}[\sqrt{\Delta}]$ and $y\in\mathbb{Z}_p[\zeta_{p^2-1}]$ we use $\bar{x}$ and $\bar{y}$ to denote the elements
$x\mod p\mathbb{Z}[\sqrt{\Delta}]$ and $y\mod p\mathbb{Z}_p[\zeta_{p^2-1}]$ respectively. If we set
$a_k=k+\sqrt{\Delta}$ for $0\le k\le p-1$, then we can view the sequence
$$S:=\overline{a_0^2},\cdots,\overline{a_0^2n^2},\cdots,\overline{a_{p-1}^2},\cdots,
\overline{a_{p-1}^2n^2}\cdots,
\overline{1^2},\cdots,\overline{n^2}$$
as a permutation $\sigma_p(g)$ of the sequence
$$S^*:=\overline{g^2},\overline{g^4},\cdots,\overline{g^{p^2-1}}.$$
We determine the sign of $\sigma_p(g)$ completely in this paper.
\end{abstract}
\maketitle

\section{Introduction}
\setcounter{lemma}{0}
\setcounter{theorem}{0}
\setcounter{corollary}{0}
\setcounter{remark}{0}
\setcounter{equation}{0}
\setcounter{conjecture}{0}
\setcounter{proposition}{0}

Investigating permutation problems in finite fields is a classical topic in number theory. First of all, many permutations on finite fields are induced by permutation polynomials over finite fields.
For instance, let $p$ be an odd prime and let $a\in\Z$ with $p\nmid a$. Clearly $f_a(x)=ax$ is a permutation polynomial over
$\F_p$. The famous Zolotarev lemma \cite{Z} says that the sign of the permutation on $\F_p$ induced by $f_a(x)$ coincides with the Legendre symbol $(\frac{a}{p})$. Also, When $k\in\Z^+$ and
$\gcd(k,p-1)=1$, the polynomial $g_k(x) = x^k$ is a permutation polynomial over $\F_p$. L.-Y Wang and the
author \cite{WW} determined the sign of this permutation induced by $g_k(x)$ by extending the method of
G. Zolotarev. In addition, W. Duke
and K. Hopkins \cite{DH} generalized this topic. They gave
the law of quadratic reciprocity on finite groups by studying the signs of some permutations induced by permutation polynomials over finite groups.

In contrast with the above, Sun \cite{S} investigated some permutations on $\F_p$ involving squares in $\F_p$. For example, let $p=2n+1$ be an odd prime and let $b_1,\cdots,b_n$ be a sequence of all the $n$ quadratic residues among $1,\cdots,p-1$ in the ascending order. Then it is easy to see that the sequence
\begin{align}\label{sequence 1}
\overline{1^2},\cdots,\overline{n^2}.
\end{align}
is a permutation $\tau_p$ on
\begin{align}\label{sequence 2}
\overline{b_1},\cdots,\overline{b_n}.
\end{align}
Here $\overline{a}$ denotes the element $a\mod p\Z$ for each $a\in\Z.$ Sun first studied this permutation and he proved that
$$\sgn(\tau_p)=\begin{cases}1&\mbox{if}\ p\equiv 3\pmod8,\\(-1)^{(h(-p)+1)/2}&\mbox{if}\ p\equiv 7\pmod8,\end{cases}$$
where $h(-p)$ is the class number of $\Q(\sqrt{-p})$ and $\sgn(\tau_p)$ is the sign of $\tau_p$. Sun also gave the explicit formula of the product
$$\prod_{1\le j<k\le \frac{p-1}{2}}(e^{2\pi ik^2/p}-e^{2\pi i j^2/p}).$$
This product has deep connections with the class number of the quadratic field $\Q(\sqrt{-p})$. Readers may see \cite{S} for details. Later the author \cite{Wu} gave the sign of $\tau_p$ in the case $p\equiv 1\pmod4$. Motivated by Sun's work, the author studied some permutations on $\F_p$ involving primitive roots modulo $p$. In fact, let $g_p\in\Z$ be a primitive root modulo $p$. Then the sequence
\begin{align}\label{sequence 3}
\overline{g_p^2},\overline{g_p^4},\cdots,\overline{g_p^{p-1}}.
\end{align}
is a permutation on the sequence (\ref{sequence 2}). In \cite{Wu} the author gave the sign of this permutation in the case $p\equiv1\pmod4$.

In view of the above, we actually investigated the permutations involving squares in $\F_p$.
Inspired by this, in this paper we mainly focus on the permutations concerning squares in
$\F_{p^2}$. We first introduce some notations and basic facts.

Let $p=2n+1$ be an odd prime, and let $\zeta_{p^2-1}$ be a primitive $(p^2-1)$-th root of unity in the algebraic closure $\overline{\Q_p}$ of $\Q_p$. By \cite[p.158 Propositon 7.12]{N} it is easy to see that $[\Q_p(\zeta_{p^2-1}):\Q_p]=2$ and that the integral closure of $\Z_p$ in $\Q_p(\zeta_{p^2-1})$ is $\Z_p[\zeta_{p^2-1}]$. Noting that $p\Z_p$ is unramified in $\Q_p(\zeta_{p^2-1})$, we therefore obtain 
$\Z_p[\zeta_{p^2-1}]/p\Z_p[\zeta_{p^2-1}]\cong \F_{p^2}$. Let $\Delta\equiv3\pmod4$ be an arbitrary quadratic non-residue modulo $p$ in $\Z$. Then clearly $p$ is inert in the field $\Q(\sqrt{\Delta})$.
Hence $\Z[\sqrt{\Delta}]/p\Z[\sqrt{\Delta}]\cong \F_{p^2}.$ Since $\Q_p(\zeta_{p^2-1})$ and $\Q_p(\sqrt{\Delta})$ are both unramified extensions of $\Q_p$ of degree $2$,
by the Local Existence Theorem (cf. \cite[p.321 Theorem 1.4]{N}) we see that
$$\Q_p(\zeta_{p^2-1})=\Q_p(\sqrt{\Delta}).$$

By the structure of the unit group of local field (cf. \cite[p.136 Proposition 5.3]{N}) we have
$$\Z_p[\zeta_{p^2-1}]^{\times}=(\zeta_{p^2-1})\times (1+p\Z_p[\zeta_{p^2-1}]).$$
Here $(\zeta_{p^2-1})=\{\zeta_{p^2-1}^k: k\in\Z\}.$
Hence we can let $g\in\Z_p[\zeta_{p^2-1}]$ be a primitive root modulo $p\Z_p[\zeta_{p^2-1}]$ with $g\equiv \zeta_{p^2-1}\pmod {p\Z_p[\zeta_{p^2-1}]}$.
For all $x\in\mathbb{Z}[\sqrt{\Delta}]$ and $y\in\mathbb{Z}_p[\zeta_{p^2-1}]$ we use the symbols
$\bar{x}$ and $\bar{y}$ to denote the elements
$x\mod p\mathbb{Z}[\sqrt{\Delta}]$ and $y\mod p\mathbb{Z}_p[\zeta_{p^2-1}]$ respectively. If we set
$a_k=k+\sqrt{\Delta}$ for $0\le k\le p-1$, then it is easy to verify that
$$\{a_k^2j^2: 0\le k\le p-1,1\le j\le n\}\cup \{j^2: 1\le j\le n\}$$ is a complete system of representatives of $(\Z[\sqrt{\Delta}]/p\Z[\sqrt{\Delta}])^{\times2}$.
We can view the sequence
\begin{align}\label{sequence 4}
S:=\overline{a_0^2},\cdots,\overline{a_0^2n^2},\cdots,\overline{a_{p-1}^2},\cdots,
\overline{a_{p-1}^2n^2}\cdots,
\overline{1^2},\cdots,\overline{n^2}
\end{align}
as a permutation $\sigma_p$ of the sequence
\begin{align}\label{sequence 5}
S^*:=\overline{g^2},\overline{g^4},\cdots,\overline{g^{p^2-1}}.
\end{align}
To state our results, we let $\beta_0\in\{0,1\}$ be the integer satisfying
\begin{equation}\label{beta}
(-1)^{\beta_0}\equiv \frac{(\sqrt{\Delta})^{\frac{p-1}{2}}}{\zeta_{p^2-1}^{\frac{p^2-1}{4}}}\pmod{p\Z_p[\zeta_{p^2-1}]}.
\end{equation}
Throughout this paper, we use the symbol $\sgn(\sigma_p(g))$ to denote the sign of $\sigma_p(g)$. Now we are in the position to state the main results of this paper.
\begin{theorem}\label{theorem A}
$$\sgn(\sigma_p(g))=\begin{cases}(-1)^{\beta_0+\frac{p+3}{4}}&\mbox{if}\ p\equiv 1\pmod4,\\(-1)^{\frac{h(-p)+1}{2}+\beta_0}&\mbox{if}\ p\equiv 3\pmod4\ \text{and}\ p>3,
\\(-1)^{1+\beta_0}&\mbox{if}\ p=3,
\end{cases}$$
where $h(-p)$ is the class number of $\Q(\sqrt{-p})$.
\end{theorem}
The proof of the above Theorem will be given in Section 2.
\maketitle

\section{Proof of the main result}
\setcounter{lemma}{0}
\setcounter{theorem}{0}
\setcounter{corollary}{0}
\setcounter{remark}{0}
\setcounter{equation}{0}
\setcounter{conjecture}{0}
\setcounter{proposition}{0}

Recall that $a_k=k+\sqrt{\Delta}$ for $k=0,1,\cdots,p-1.$ We need the following several lemmas involving $a_k$. For convenience, we write $p\Z[\sqrt{\Delta}]=\mathfrak{p}$.

\begin{lemma}\label{lemma ak A}
Let $A_p=\prod_{0\le k\le p-1}a_k$. Then we have
$$A_p^{\frac{(p-1)(p-3)}{4}}\equiv\begin{cases}\Delta^{-\frac{p-1}{4}}\pmod {\mathfrak{p}}&\mbox{if}\ p\equiv 1\pmod4,\\(-1)^{\frac{p-3}{4}}\pmod {\mathfrak{p}}&\mbox{if}\ p\equiv 3\pmod4.\end{cases}$$
\end{lemma}
\begin{proof}
Since
$$\prod_{0\le t\le p-1}(x+t)\equiv x^p-x\pmod {p\Z[x]},$$
we have
$$A_p^{\frac{(p-1)(p-3)}{4}}=\prod_{0\le t\le p-1}(\sqrt{\Delta}+t)^{\frac{(p-1)(p-3)}{4}}
\equiv (-2\sqrt{\Delta})^{\frac{(p-1)(p-3)}{4}}\pmod {\mathfrak{p}}.$$
Observing that $(\sqrt{\Delta})^{p-1}\equiv-1\pmod {\mathfrak{p}}$, one may easily get the desired result.
\end{proof}

\begin{lemma}\label{lemma ak B}
Let $B_p=\prod_{0\le k\le p-1}(1-a_k^{p-1})$. Then we have
$$B_p^{\frac{p-1}{2}}\equiv 1\pmod {\mathfrak{p}}.$$
\end{lemma}
\begin{proof}
For each $k=0,\cdots,p-1$ we have
\begin{align}\label{equation congruence}
a_k^p=(k+\sqrt{\Delta})^p\equiv k+(\sqrt{\Delta})^{p-1}\sqrt{\Delta}\equiv k-\sqrt{\Delta}\pmod {\mathfrak{p}}.
\end{align}
Hence we have the following congruences
\begin{align*}
B_p^{\frac{p-1}{2}}\equiv&\prod_{0\le k\le p-1}\(1-\frac{k-\sqrt{\Delta}}{k+\sqrt{\Delta}}\)^{\frac{p-1}{2}}
\\=& 2^{p\frac{p-1}{2}}(\sqrt{\Delta})^{\frac{(p-1)^2}{2}}\prod_{1\le k\le \frac{p-1}{2}}\(\frac{1}{k+\sqrt{\Delta}}\)^{\frac{p-1}{2}}\(\frac{1}{p-k+\sqrt{\Delta}}\)^{\frac{p-1}{2}}
\\\equiv& \(\frac{-2}{p}\)\prod_{1\le k\le \frac{p-1}{2}}\(\frac{1}{\Delta-k^2}\)^{\frac{p-1}{2}}\pmod{\mathfrak{p}}.
\end{align*}
Noting that
\begin{align}\label{equation quadratic residues}
\prod_{1\le k\le \frac{p-1}{2}}(x-k^2)\equiv x^{\frac{p-1}{2}}-1\pmod {p\Z[x]},
\end{align}
we obtain 
$$\prod_{1\le k\le \frac{p-1}{2}}\(\frac{1}{\Delta-k^2}\)^{\frac{p-1}{2}}\equiv \(\frac{-2}{p}\)\pmod {\mathfrak{p}}.$$
Hence
$$B_p^{\frac{p-1}{2}}\equiv 1\pmod {\mathfrak{p}}.$$
\end{proof}

\begin{lemma}\label{lemma ak C}
Let $C_p=\prod_{1\le s<t\le p-1}\frac{1}{(t+\sqrt{\Delta})(s+\sqrt{\Delta})}$. Then
$$C_p^{\frac{p-1}{2}}\equiv \(\frac{-2}{p}\)\pmod {\mathfrak{p}}.$$
\end{lemma}
\begin{proof}
Clearly we have
\begin{align*}
C_p=&\prod_{1\le s<t\le \frac{p-1}{2}}\frac{1}{(t+\sqrt{\Delta})(s+\sqrt{\Delta})}
\frac{1}{(p-t+\sqrt{\Delta})(p-s+\sqrt{\Delta})}\\\times
&\prod_{1\le s\le \frac{p-1}{2}}\prod_{1\le t\le \frac{p-1}{2}}\frac{1}{(p-t+\sqrt{\Delta})(s+\sqrt{\Delta})}.
\end{align*}
Hence we obtain 
\begin{align*}
C_p^{\frac{p-1}{2}}\equiv& \prod_{1\le s<t\le \frac{p-1}{2}}\(\frac{\Delta-t^2}{p}\)\(\frac{\Delta-s^2}{p}\)
\\\times&\prod_{1\le s\le \frac{p-1}{2}}\prod_{1\le t\le \frac{p-1}{2}}\(\frac{1}{(\sqrt{\Delta}-t)(\sqrt{\Delta}+s)}\)^{\frac{p-1}{2}}\pmod{\mathfrak{p}}.
\end{align*}
We first handle the product
$$\prod_{1\le s\le \frac{p-1}{2}}\prod_{1\le t\le \frac{p-1}{2}}\(\frac{1}{(\sqrt{\Delta}-t)(\sqrt{\Delta}+s)}\)^{\frac{p-1}{2}}\pmod{\mathfrak{p}}.$$
Noting that
\begin{align*}
\prod_{1\le s\le \frac{p-1}{2}}(x+s)\prod_{1\le t\le \frac{p-1}{2}}(x-t)\equiv x^{p-1}-1\pmod{p\Z[x]},
\end{align*}
we therefore get that
$$\prod_{1\le t\le \frac{p-1}{2}}(\sqrt{\Delta}-t)\equiv \frac{-2}{\prod_{1\le s\le \frac{p-1}{2}}(\sqrt{\Delta}+s)}\pmod{\mathfrak{p}}.$$
Hence
\begin{align}\label{equation C}
\prod_{1\le s\le \frac{p-1}{2}}\prod_{1\le t\le \frac{p-1}{2}}\(\frac{1}{(\sqrt{\Delta}-t)(\sqrt{\Delta}+s)}\)^{\frac{p-1}{2}}
\equiv\(\frac{-2}{p}\)^{\frac{p-1}{2}}\pmod{\mathfrak{p}}.
\end{align}
We now turn to the product
$$\prod_{1\le s<t\le \frac{p-1}{2}}\(\frac{\Delta-t^2}{p}\)\(\frac{\Delta-s^2}{p}\).$$
One can easily verify the following identities
\begin{align}\label{equation sums A}
&\#\{(x^2,y^2): 1\le x,y\le\frac{p-1}{2},x^2+y^2\equiv \Delta\pmod p\}\\&=
\begin{cases}\frac{p-1}{4}&\mbox{if}\ p\equiv 1\pmod4,\\\frac{p+1}{4}&\mbox{if}\ p\equiv 3\pmod4.\end{cases}
\end{align}
and
\begin{align}\label{equation sums B}
&\#\{(x^2,y^2): 1\le x,y\le\frac{p-1}{2},x^2+\Delta y^2\equiv \Delta\pmod p\}\\&=
\begin{cases}\frac{p-1}{4}&\mbox{if}\ p\equiv 1\pmod4,\\\frac{p-3}{4}&\mbox{if}\ p\equiv 3\pmod4.\end{cases}
\end{align}
From the above we see that
\begin{align*}
&\#\{(s,t): 1\le s<t\le \frac{p-1}{2}: \(\frac{\Delta-t^2}{p}\)\(\frac{\Delta-s^2}{p}\)=-1\}
\\&=\begin{cases}\frac{(p-1)^2}{16}&\mbox{if}\ p\equiv 1\pmod4,\\\frac{p-3}{4}\cdot\frac{p+1}{4}&\mbox{if}\ p\equiv 3\pmod4.\end{cases}
\end{align*}
Therefore

\begin{align}\label{equation legendre}
\prod_{1\le s<t\le \frac{p-1}{2}}\(\frac{\Delta-t^2}{p}\)\(\frac{\Delta-s^2}{p}\)=
\begin{cases}(-1)^{\frac{p-1}{4}}&\mbox{if}\ p\equiv 1\pmod4,\\1&\mbox{if}\ p\equiv 3\pmod4.\end{cases}
\end{align}
Then our desired result follows from (\ref{equation C}) and (\ref{equation legendre}).
\end{proof}

\begin{lemma}\label{lemma ak D}
Let $D_p=\prod_{0\le s<t\le p-1}(a_t^{p-1}-a_s^{p-1})$. Then $D_p^{\frac{p-1}{2}}\pmod {\mathfrak{p}}$ is equal to
$$\begin{cases}(\sqrt{\Delta})^{-\frac{(p-1)^2}{4}}\pmod{\mathfrak{p}}&\mbox{if}\ p\equiv 1\pmod4,
\\(\sqrt{\Delta})^{-\frac{(p-1)^2}{4}}(-1)^{\frac{h(-p)+1}{2}}\cdot(\frac{2}{p})\pmod{\mathfrak{p}}&\mbox{if}\ p\equiv 3\pmod4\ \text{and}\ p>3,\\-(\sqrt{\Delta})^{-1}\pmod{\mathfrak{p}}&\mbox{if}\ p=3.
\end{cases}$$
\end{lemma}
\begin{proof}
From (\ref{equation congruence}) one may easily verify that
$D_p^{\frac{p-1}{2}}\pmod {\mathfrak{p}}$ is equal to
\begin{align*}
\(\frac{t-\sqrt{\Delta}}{t+\sqrt{\Delta}}-\frac{s-\sqrt{\Delta}}{s+\sqrt{\Delta}}\)^{\frac{p-1}{2}}
\equiv\prod_{0\le s<t\le p-1}\(\frac{2\sqrt{\Delta}(t-s)}{(t+\sqrt{\Delta})(s+\sqrt{\Delta})}\)^{\frac{p-1}{2}}
\pmod{\mathfrak{p}}
\end{align*}
From this we further obtain that the above is equal to
\begin{align*}
\(\frac{-2}{p}\)^{\frac{p+1}{2}}\(\frac{-1}{\sqrt{\Delta}}\)^{\frac{(p-1)^2}{4}}C_p^{\frac{p-1}{2}}
\prod_{0<t<p}\(\frac{1}{t+\sqrt{\Delta}}\)^{\frac{p-1}{2}}\prod_{0<s<t<p}(t-s)^{\frac{p-1}{2}} \pmod{\mathfrak{p}}.
\end{align*}
We first handle the product
$$\prod_{1\le t\le p-1}\(\frac{1}{t+\sqrt{\Delta}}\)^{\frac{p-1}{2}}.$$
By (\ref{equation quadratic residues}) we have
\begin{align}\label{equation D}
\prod_{1\le t\le p-1}\(\frac{1}{t+\sqrt{\Delta}}\)^{\frac{p-1}{2}}\equiv
\prod_{1\le t\le\frac{p-1}{2}}\(\frac{1}{\Delta-t^2}\)^{\frac{p-1}{2}}
\equiv\(\frac{-2}{p}\)\pmod{\mathfrak{p}}.
\end{align}
We turn to the product
$$\prod_{1\le s<t\le p-1}(t-s)^{\frac{p-1}{2}}.$$
It is clear that
$$\prod_{1\le s<t\le p-1}(t-s)^{\frac{p-1}{2}}\pmod{\mathfrak{p}}$$
is equal to
\begin{align*}\label{equation E}
&\prod_{1\le s<t\le\frac{p-1}{2}}\(\frac{t-s}{p}\)\(\frac{-s+t}{p}\)
\prod_{1\le s\le\frac{p-1}{2}}\prod_{1\le t\le\frac{p-1}{2}}\(\frac{-1}{p}\)\(\frac{t+s}{p}\)
\\&\equiv(-1)^{\frac{p-1}{2}}\prod_{1\le s\le\frac{p-1}{2}}\prod_{1\le t\le\frac{p-1}{2}}\(\frac{t+s}{p}\)
\pmod{\mathfrak{p}}.
\end{align*}
We now divide our proof into the following two cases.

{\it Case} 1. $p\equiv1\pmod4$.

Let $1\le w\le \frac{p-1}{2}$ be an arbitrary quadratic non-residue modulo $p$. Then
$$\#\{(s,t): 1\le s,t\le\frac{p-1}{2}, s+t\equiv w\pmod p\}=w-1,$$
and
$$\#\{(s,t): 1\le s,t\le\frac{p-1}{2}, s+t\equiv p-w\pmod p\}=w.$$
Hence when $p\equiv 1\pmod 4$ we have
\begin{equation}\label{equation F}
\prod_{1\le s\le \frac{p-1}{2}}\prod_{1\le t\le \frac{p-1}{2}}\(\frac{t+s}{p}\)
=(-1)^{\#\{1\le w\le\frac{p-1}{2}:(\frac{w}{p})=-1\}}=(-1)^{\frac{p-1}{4}}
\end{equation}

{\it Case} 2. $p\equiv 3\pmod4$.

Let $1\le w\le \frac{p-1}{2}$ be an arbitrary quadratic non-residue modulo $p$, and let
$1\le v\le \frac{p-1}{2}$ be an arbitrary quadratic residue modulo $p$.
Then
$$\#\{(s,t): 1\le s,t\le\frac{p-1}{2}, s+t\equiv w\pmod p\}=w-1,$$
and
$$\#\{s,t): 1\le s,t\le\frac{p-1}{2}, s+t\equiv p-v\pmod p\}=v.$$
Hence
\begin{equation*}
\prod_{1\le s\le \frac{p-1}{2}}\prod_{1\le t\le \frac{p-1}{2}}\(\frac{t+s}{p}\)
=(-1)^{\#\{1\le w\le\frac{p-1}{2}: (\frac{w}{p})=-1\}}\cdot(-1)^{\frac{p^2-1}{8}}.
\end{equation*}
For each $p\equiv3\pmod4$ let $h(-p)$ be the class number formula of $\Q(\sqrt{-p})$. When $p>3$, by the
class number formula we have
$$(2-\(\frac{2}{p}\))h(-p)=\frac{p-1}{2}-2\#\{1\le w\le\frac{p-1}{2}: \(\frac{w}{p}\)=-1\}.$$
From this one may easily verify that
$$\#\{1\le w\le\frac{p-1}{2}: \(\frac{w}{p}\)=-1\}\equiv \frac{h(-p)+1}{2}\pmod2.$$
The readers may also see Mordell's paper \cite{M} for details.

From the above, we obtain
\begin{equation}\label{equation G}
\prod_{1\le s\le \frac{p-1}{2}}\prod_{1\le t\le \frac{p-1}{2}}\(\frac{t+s}{p}\)
=\begin{cases}(-1)^{\frac{h(-p)+1}{2}}\cdot(\frac{2}{p})&\mbox{if}\ p\equiv 3\pmod4\ \text{and}\ p>3,
\\-1&\mbox{if}\ p=3.\end{cases}
\end{equation}
In view of the above, we obtain the desired result.
\end{proof}

We let $\Phi_{p^2-1}(x)\in\Z[x]$ denote the $(p^2-1)$-th cyclotomic polynomial. We also let
$$F(x)=\prod_{1\le s<t\le (p^2-1)/2}(x^{2t}-x^{2s}),$$
and let
$$T(x)=(-1)^{\frac{p^2+7}{8}}\(\frac{p^2-1}{2}\)^{\frac{p^2-1}{4}}\cdot x^{\frac{(p^2-1)}{4}}\in\Z[x].$$
We need the following lemma. We also set $\zeta=e^{2\pi i/(p^2-1)}$.
\begin{lemma}\label{lemma cyclotomic}
$\Phi_{p^2-1}(x)\mid F(x)-T(x)$ in $\Z[x]$.
\end{lemma}
\begin{proof}
It is sufficient to prove that $F(\zeta)=T(\zeta)$.
We first compute $F(\zeta)^2$. We have the following equalities:
\begin{align*}
F(\zeta)^2=&\prod_{1\le s<t\le \frac{p^2-1}{2}}(\zeta^{2t}-\zeta^{2s})^2
\\=&(-1)^{\frac{(p^2-1)(p^2-3)}{8}}\cdot\prod_{1\le s\ne t\le \frac{p^2-1}{2}}(\zeta^{2t}-\zeta^{2s})
\\=&\prod_{1\le t\le\frac{p^2-1}{2}}\frac{x^{\frac{p^2-1}{2}}-1}{x-\zeta^{2t}}\Big|_{x=\zeta^{2t}}
\\=&\(\frac{p^2-1}{2}\)^{\frac{p^2-1}{2}}\prod_{1\le t\le \frac{p^2-1}{2}}\zeta^{-2t}=-1\cdot\(\frac{p^2-1}{2}\)^{\frac{p^2-1}{2}}.
\end{align*}
Hence $F(\zeta)=\pm i\cdot(\frac{p^2-1}{2})^{\frac{p^2-1}{2}}$. We now compute the argument of
$F(\zeta)$. Noting that for any $1\le s<t\le (p^2-1)/2$ we have
$$\zeta^{2t}-\zeta^{2s}=
\zeta^{t+s}(\zeta^{t-s}-\zeta^{-(t-s)}),$$
we therefore obtain
$$\Arg(\zeta^{2t}-\zeta^{2s})=\frac{2\pi}{p^2-1}(t+s)+\frac{\pi}{2}.$$
From this we have
\begin{align*}
\Arg (F(\zeta))=&\sum_{1\le s<t\le\frac{p^2-1}{2}}\(\frac{2\pi}{p^2-1}(t+s)+\frac{\pi}{2}\)
\\=&\frac{(p^2-1)(p^2-3)\pi}{16}+\frac{2\pi}{p^2-1}\cdot\sum_{1\le s<t\le \frac{p^2-1}{2}}(t+s)
\\\equiv&-\frac{\pi}{2}+\frac{p^2-1}{8}\pi\pmod{2\pi\Z}.
\end{align*}
Therefore
$$F(\zeta)=i(-1)^{\frac{p^2+7}{8}}\(\frac{p^2-1}{2}\)^{\frac{p^2-1}{4}}
=T(\zeta).$$
This completes the proof.
\end{proof}

We are now in the position to prove Theorem \ref{theorem A}.

\noindent{\bf Proof of Theorem\ \ref{theorem A}}.
Let $S=\{\alpha_1,\cdots,\alpha_n\}$ be a finite subset of a finite field and let $\tau$ be a permutation on $S$. Then it follows from definition that the sign of $\tau$ denoted by $\sgn({\tau})$ is
$$\prod_{1\le s<t\le n}\frac{\tau(\alpha_t)-\tau(\alpha_s)}{\alpha_t-\alpha_s}.$$
From this we see that
$$\sgn(\sigma_p)=\prod_{1\le s<t\le \frac{p^2-1}{2}}
\frac{\overline{g^{2t}}-\overline{g^{2s}}}{\sigma_p(\overline{g^{2t}})-\sigma_p(\overline{g^{2s}})}.$$
We first handle the numerator. For convenience, we set $\mathfrak{B}=p\Z_p[\zeta_{p^2-1}]$.
Clearly $\Phi_{p^2-1}(x) \mod p\Z_p[\zeta_{p^2-1}][x]$ splits completely in $\Z_p[\zeta_{p^2-1}]/\mathfrak{B}[x].$  As $g\equiv \zeta_{p^2-1}\pmod{\mathfrak{B}}$
by Lemma \ref{lemma cyclotomic} we see that
$$\prod_{1\le s<t\le \frac{p^2-1}{2}}(g^{2t}-g^{2s})\pmod{\mathfrak{B}}$$
is equal to
\begin{align}\label{equation numerator}
-\(\frac{2}{p}\)\(\frac{p^2-1}{2}\)^{\frac{p^2-1}{4}}g^{\frac{p^2-1}{4}}
\equiv -\(\frac{2}{p}\)\(\frac{-2}{p}\)^{\frac{p+1}{2}}g^{\frac{p^2-1}{4}}
\pmod{\mathfrak{B}}.
\end{align}
We now turn to the denominator. It is easy to verify that
$$\prod_{1\le s<t\le\frac{p^2-1}{2}}(\sigma_p(g^{2t})-\sigma_p(g^{2s}))\pmod {\mathfrak{p}}$$
is equal to
$$A_p^{\frac{(p-1)(p-3)}{4}}B_p^{\frac{p-1}{2}}D_p^{\frac{p-1}{2}}
\prod_{1\le s<t\le\frac{p-1}{2}}(t^2-s^2)^2\pmod {\mathfrak{p}}.$$
By \cite[(1.5)]{S} we have
\begin{equation}\label{equation H}
\prod_{1\le s<t\le\frac{p-1}{2}}(t^2-s^2)^2\equiv (-1)^{\frac{p+1}{2}}\pmod p.
\end{equation}
By the above, we obtain that
$$\prod_{1\le s<t\le\frac{p^2-1}{2}}(\sigma_p(g^{2t})-\sigma_p(g^{2s}))\pmod {\mathfrak{p}}$$
is equal to
\begin{align}\label{equation denominator}
\begin{cases}-\Delta^{-\frac{p-1}{4}}(\sqrt{\Delta})^{-\frac{(p-1)^2}{4}}\pmod{\mathfrak{p}}&\mbox{if}\ p\equiv 1\pmod4,\\
(-1)^{\frac{h(-p)-1}{2}}(\sqrt{\Delta})^{-\frac{(p-1)^2}{4}}\pmod{\mathfrak{p}}&\mbox{if}\ p\equiv 3\pmod4\ \text{and}\ p>3,
\\-(\sqrt{\Delta})^{-1}\pmod{\mathfrak{p}}&\mbox{if}\ p=3.\end{cases}
\end{align}
Let $\sqrt{\Delta}\equiv \zeta_{p^2-1}^{\alpha}\pmod{\mathfrak{B}}$ for some $\alpha\in\Z$. Since
$(\sqrt{\Delta})^{p-1}\equiv-1\pmod{\mathfrak{B}},$ we obtain
$$(p-1)\alpha\equiv \frac{p^2-1}{2}\pmod{p^2-1}.$$
Hence
$$\alpha\equiv\frac{p+1}{2}\pmod {p+1}.$$
We set $\alpha=\frac{p+1}{2}+(p+1)\beta$ for some $\beta\in\Z$. Then we have
$$(\sqrt{\Delta})^{\frac{p-1}{2}}\equiv \zeta_{p^2-1}^{\frac{p^2-1}{4}}\zeta_{p^2-1}^{\frac{p^2-1}{2}\beta}
\pmod{\mathfrak{B}}.$$
From this we get
$$(-1)^{\beta}\equiv \frac{(\sqrt{\Delta})^{\frac{p-1}{2}}}{\zeta_{p^2-1}^{\frac{p^2-1}{4}}}\pmod{\mathfrak{B}}.$$
Therefore $\beta\equiv\beta_0\pmod2$, where $\beta_0$ is defined as in (\ref{beta}). We divide the remaining  proof into three cases.

{\it Case} 1. $p\equiv1\pmod4.$

By (\ref{equation numerator}) and (\ref{equation denominator}) we have
$$\sgn(\sigma_p)\equiv g^{\frac{p^2-1}{4}+\frac{p-1}{2}\alpha+\frac{(p-1)^2}{4}\alpha}\pmod{\mathfrak{B}}.$$
Replacing $\alpha$ by $\frac{p+1}{2}+(p+1)\beta$ and noting that $g^{\frac{p^2-1}{2}}\equiv-1\pmod{\mathfrak{B}}$, we obtain that when $p\equiv1\pmod4$
$$\sgn(\sigma_p)=(-1)^{\beta_0+\frac{p+3}{4}}.$$

{\it Case} 2. $p\equiv3\pmod4$ and $p>3$.

Similar to the Case 1, we have
$$\sgn(\sigma_p)\equiv \(\frac{2}{p}\)g^{\frac{p^2-1}{4}}(-1)^{\frac{h(-p)+1}{2}}g^{\frac{(p-1)^2}{4}\alpha}\pmod{\mathfrak{B}}.$$
Then via computation we obtain that
$$\sgn(\sigma_p)=(-1)^{\frac{h(-p)+1}{2}+\beta_0}.$$

{\it Case} 3. $p=3$.

In this case it is easy to see that
$$\sgn(\sigma_3)=(-1)^{1+\beta_0}.$$

In view of the above, we complete the proof.\qed

\Ack\ This research was supported by the National Natural Science
Foundation of China (Grant No. 11971222).

\end{document}